\documentclass[11pt,leqno]{amsart}
\usepackage{amsthm,amsfonts,amssymb,amsmath,oldgerm}
\numberwithin{equation}{section}
\usepackage{fullpage}
\usepackage{color}
\usepackage{calrsfs}
\DeclareMathAlphabet{\pazocal}{OMS}{zplm}{m}{n}

\newcommand\R{\mathbb R}

\newcommand\kernel{\hbox{\rm Ker}}

\newcommand\br{\begin{remark}}
\newcommand\er{\end{remark}}
\newcommand\bp{\begin{pmatrix}}
\newcommand\ep{\end{pmatrix}}
\newcommand{\be}{\begin{equation}}
\newcommand{\ee}{\end{equation}}
\newcommand\ba{\begin{equation}\begin{aligned}}
\newcommand\ea{\end{aligned}\end{equation}}

\newcommand{\bap}{\begin{app}}
\newcommand{\eap}{\end{app}}
\newcommand{\begs}{\begin{exams}}
\newcommand{\eegs}{\end{exams}}
\newcommand{\beg}{\begin{example}}
\newcommand{\eeg}{\end{exaplem}}
\newcommand{\bpr}{\begin{proposition}}
\newcommand{\epr}{\end{proposition}}
\newcommand{\bt}{\begin{theorem}}
\newcommand{\et}{\end{theorem}}
\newcommand{\bc}{\begin{corollary}}
\newcommand{\ec}{\end{corollary}}
\newcommand{\bl}{\begin{lemma}}
\newcommand{\el}{\end{lemma}}
\newcommand{\bd}{\begin{definition}}
\newcommand{\ed}{\end{definition}}
\newcommand{\brs}{\begin{remarks}}
\newcommand{\ers}{\end{remarks}}

\newtheorem{theorem}{Theorem}[section]
\newtheorem{proposition}[theorem]{Proposition}
\newtheorem{corollary}[theorem]{Corollary}
\newtheorem{lemma}[theorem]{Lemma}

\theoremstyle{remark}
\newtheorem{remark}[theorem]{Remark}
\theoremstyle{definition}
\newtheorem{definition}[theorem]{Definition}

\newtheorem{example}[theorem]{Example}

\newcommand\cL{{\mathcal L}}

\newcommand{\beq}{\begin{equation}}
\newcommand{\eeq}{\end{equation}}

\title{$L^\infty$ resolvent bounds for steady Boltzmann's equation} 
\author{Kevin Zumbrun}
\address{Indiana University, Bloomington, IN 47405}
\email{kzumbrun@indiana.edu}
\thanks{Research of K.Z. was partially supported under NSF grant no. DMS-0300487}

\begin{document}

\begin{abstract}
We derive lower bounds on the resolvent operator for the linearized steady Boltzmann equation over weighted $L^\infty$
Banach spaces in velocity, comparable to those derived by Pogan\&Zumbrun in an analogous weighted $L^2$ Hilbert space setting.
These show in particular that the operator norm of the resolvent kernel is unbounded in $L^p(\R)$ for all $1<p \leq \infty$,
resolving an apparent discrepancy in behavior between the two settings
suggested by previous work.
\end{abstract}

\maketitle

\section{Introduction}\label{s:intro}
In this note, we derive an $L^\infty$ {\it lower} bound on the resolvent 
operator for the linearized steady Boltzmann equation,
in the process resolving a discrepancy between results of \cite{LY2} and \cite{PZ1}.
Let $L^\infty_{r,\xi}$ denote the space of functions $h(\xi)$, $\xi\in \R^3$, with finite norm 
$ \|h\|_{L^\infty_{r,\xi}}:=\sup_{\xi\in \R^3}(1+|\xi|)^r |h(\xi)|.  $
A key element in the study \cite{LY2} of existence of invariant manifolds for the steady Boltzmann equation 
with hard sphere potential is the resolvent estimate \cite[Thm. 11, (108)]{LY2}
\be\label{bound}
\|h\|_{L^\infty_{3,\xi}}(x) \leq C\int_\R e^{-\beta |x-y|}\|g\|_{L^\infty_{2,\xi}}(y)dy,
\quad \beta>0
\ee
for solutions of the linearized inhomogeneous equation with ``microscopic'' data, 
\be\label{inhom}
\xi_1 \partial_x h- L h= g, 
\qquad
g\in \kernel L^\perp,
\ee
where $L:=\bar M^{-1/2}\cL \bar M^{1/2}$,
with $\cL$ the linearized collision operator about a reference Maxwellian $ \bar M(\xi)=ce^{-|\xi-v|^2/d} $,
and $h :=(f-\bar M)/\bar M^{1/2}$,
with $f(x,\xi,t)$ 
the standard Boltzmann variable denoting density of particles of velocity $\xi\in \R^3$ at point 
$(x,t)\in \R^2$ in space and time \cite{G,LY2}.\footnote{
Here, $(\cdot)^\perp$ denotes orthogonal complement with respect to the usual, unweighted space 
$L^2\supset L^\infty_{2,\xi}$ in $\xi$.}

Estimate \eqref{bound} may be expressed equivalently as a bound
$
|R(x,\cdot, \cdot)|_{L^\infty_{2,\xi}\to L^\infty_{3,\xi}}\leq Ce^{-\beta |x|}$,
$\beta>0 $
on the resolvent kernel $R(x,\xi,\xi_*)$ for \eqref{inhom}, 
that is, the kernel of resolvent operator 
$$
\mathcal{R}=(\xi_1\partial_x-L)^{-1}|_{\kernel L^\perp},
\qquad
(\mathcal{R}g)(x,\xi)=\int_\R \int_{\R^3} R(x-y,\xi,\xi_*)g(y,\xi_*) d\xi_* dy.
$$
(See \cite{LY2} for a construction of $R$.) 
Either implies by standard convolution bounds uniform estimates
\be\label{uni}
\|h\|_{L^\infty(\R,L^\infty_{3,\xi})} \leq C\|g\|_{L^q(\R, L^\infty_{2,\xi})}, \qquad 1\leq q\leq \infty.
\ee
Taking $q=\infty$ in \eqref{uni},
and noting \cite{G,MZ1} that the associated bilinear collision operator $Q(h,h)$ of the full, nonlinear equation 
is bounded from $L^\infty_{3,\xi}\times L^\infty_{3,\xi} \to L^\infty_{2,\xi}$, 
yields $L^\infty(\R,L^\infty_{3,\xi})$-contractivity of the map $h\to \mathcal{R}Q(h,h)$,
giving the basis for the fixed-point iteration schemes defined in \cite{LY2}.

By comparison, in a nearby Hilbert-space setting,\footnote{
Precisely, the $L^2$ norm in $\xi$, after the rescaling $h\to \langle \xi\rangle^{1/2} h$, $g\to \langle \xi\rangle^{-1/2} g$, 
where $\langle \xi\rangle:=(1+|\xi|^2)^{1/2}$; see \cite{MZ1,PZ1}.  }
it has been shown \cite{PZ1} that the 
operator norm of the resolvent kernel is unbounded in all $L^p(\R)$, 
by explicit computation using self-adjointness of $L$ to diagonalize by unitary transformation.
This is in sharp contrast with \eqref{bound}, 
especially since finite-dimensional intuition
would suggest that optimal bounds be found in coordinates for which 
$L$ is unitarily diagonalizable.
Here, we resolve this apparent paradox by (i) {\it establishing the following result 
contradicting \eqref{uni},} hence \eqref{bound}, and 
(ii) {\it identifying a corresponding error in \cite{LY2}}.

\begin{proposition}\label{main}
The solution operator for \eqref{inhom} does not satisfy 
$ \|h\|_{L^\infty(\R,L^\infty_{3,\xi})} \leq C \|g\|_{L^q(\R,L^\infty_{2,\xi})} $ for any $q<\infty$, $C=C(q)$.
Likewise, $|R(x,\cdot,\cdot)|_{L^\infty_{2,\xi} \to L^\infty_{3,\xi}}$ is unbounded in $L^p(\R)$, $p>1$.
\end{proposition}

\begin{proof}
Recall \cite{G} that the linearized collision operator $L$ appearing in \eqref{inhom}
may be decomposed as 
$ L=-{\nu}(\xi) + {K}, $
where $ \nu(\xi)\sim \langle\xi\rangle :=(1+|\xi|^2)^{1/2}$ and 
$ ({K}h)(\xi)=\int_{\R^3} k(\xi,\xi_*)h(\xi_*)d\xi_*, $
with kernel $ |k(\xi,\xi_*)|\leq C|\xi-\xi_*|^{-1}e^{-c|\xi-\xi_*|^2}.  $
By $\||\xi|^{-1}e^{-c|\xi|^2}\|_{L^1}<\infty$ and standard convolution bounds, 
${K}$ is bounded on $L^\infty(\xi)$.
Similarly, using the inequality $\langle \xi\rangle/\langle \xi-\xi_* \rangle\leq C\langle \xi_*\rangle$ and
$\||\xi|^{r-1}e^{-c|\xi|^2}\|_{L^1}<\infty$,
we obtain the standard result $|K|_{L^\infty_{r,\xi}}<+\infty$ \cite{G,LY2}.
Recall, further, that $\kernel L$ is finite-dimensional, tangent to the
5-dimensional manifold of Maxwellians \cite{G}.
Defining $\mathcal{S}=\big( \xi_1\partial_x + \nu(\xi) \big)^{-1}$, 
we have 
the explicit solution 
formula
$ (\mathcal{S}g)(x,\xi)= \int_{\R}S_\xi(x-y)g(y,\xi) dy, $
with $ S_\xi(\theta)=\xi_1^{-1}e^{-\nu(\xi)\xi_1^{-1}|\theta|}$ for  $ \theta \xi_1>0$, $0$ otherwise.
Computing $ \|S_\xi(\cdot)\|_{L^1(\R)}= 1/\nu(\xi)\leq C/ \langle \xi\rangle $,
we have by standard convolution bounds 
$|\mathcal{S}|_{L^\infty(\R, L^\infty_{2,\xi})\to  L^\infty(\R, L^\infty_{3,\xi})}<+\infty$.

Writing \eqref{inhom} as $\big( \xi_1 \partial_x + \nu(\xi) \big)h= {K}h + g$, applying 
$\mathcal{S}$, and rearranging, gives
$ \mathcal{S}g= h- \mathcal{S}{K}h, $
hence $ | \mathcal{S}g |_{L^\infty(\R,L^\infty_{3,\xi})}\leq C| h |_{L^\infty(3,L^\infty_{r,\xi})} $
by boundedness of $|K|_{L^\infty_{3,\xi}}$, $|\mathcal{S}|_{L^\infty(\R,L^\infty_{3,\xi})}$.
Thus, \eqref{uni} would imply a comparable bound 
$
|\mathcal{S}g|_{L^\infty(\R,L^\infty_{3,\xi})}\leq C |g|_{L^q(\R,L^\infty_{2,\xi})}
$
on $\mathcal{S}$.  This is easily seen to be false for $g \in L^\infty_{2,\xi}$, using test functions 
$ g_\alpha(x,\xi)=\langle \xi\rangle^{-2}\alpha^{-1/q}e^{-\alpha^{-1} |x|}$,
$\|g_\alpha\|_{L^q(\R,L^\infty_{2,\xi})}\equiv 1$, $\alpha\to 0 $, and computing
$ \mathcal{S}g_\alpha(0,(\alpha,1,0))=
(1+\alpha^2)^{-1}
\int_0^{+\infty} \alpha^{-1}e^{-(1+\alpha^2)^{1/2}\alpha^{-1}y} \alpha^{-1/q}e^{-\alpha^{-1} y}dy
\sim
\alpha^{-1/q} \to +\infty.  $
Using linear combinations of members $g_{\alpha_j}$
of the infinite-dimensional family $\{g_\alpha\}$,
we may obtain a contradiction also for $g$ ranging on any finite-codimension subspace,
in particular on $\kernel L^\perp$.  
\end{proof}

{\bf The argument of \cite{LY2}.}
In \cite{LY2}, the authors study the steady Boltzmann equation 
by a time-regularization scheme based on detailed pointwise bounds
on the time-evolutionary Green function $G(x,t,\xi,\xi_*)$, through the relation $R=\int_{\R^+}G \, dt$.
{However, this analysis contains a key error.
By 
\cite[(67)]{LY2}:
$ S_xb=\int_{\R^+}\int_{\R^3} G(x,t,\cdot,\xi_*) \xi_{1,*} b(\xi_*) \, d\xi_* dt $ 
and the symmetric 
\cite[(68)]{LY2},
\cite[(52)]{LY2} is equivalent to
	$ \|\int_{\R^+}\int_{\R^3} G(x,t,\cdot,\xi_*) \xi_{1,*} b(\xi_*)\, d\xi_* dt\|_{L^\infty_{3,\xi}} \leq Ce^{-\beta|x|}\|b\|_{L^\infty_{3,\xi}}.  $
In 
\cite[(102)]{LY2}, 
this is used to estimate 
\be\label{102}
\|\int_{\R^+}\int_{\R^3} G(x,t,\cdot,\xi_*)  \tilde b(\xi_*) \, d\xi_* dt\|_{L^\infty_{3,\xi}} \leq Ce^{-\beta|x|}\|\tilde b\|_{L^\infty_{2,\xi}}, 
\ee
in effect asserting $\|\xi_1^{-1}\tilde b\|_{L^\infty_{3,\xi}} \lesssim \| \tilde b\|_{L^\infty_{2,\xi}}$,
or $|\xi_1|^{-1}\lesssim (1+ |\xi|)^{-1}$-- {evidently false} for $\xi_1$ small.
But, \eqref{102} is the basis for \eqref{bound},
the resolvent estimate underlying the fixed-point constructions of \cite{LY2}.

{\bf Discussion and open problems.}
Proposition \ref{main} shows that \cite[Thm. 11]{LY2} is incorrect, 
invalidating the conclusions of \cite{LY2} on existence of invariant manifolds.
Our proof of Proposition \ref{main} depends on the property that the principle part $\mathcal{S}$ of the resolvent $\mathcal{R}$
is unbounded from $L^q(x,L^\infty_{2,\xi})\to L^\infty (x,L^\infty_{3,\xi})$ for $q<\infty$, which is
much stronger than $|S(x,\cdot, \cdot)|_{L^\infty_{2,\xi}\to L^\infty_{3,\xi}}\not \in L^{p}$ for $1=1/p+1/q$.  
Estimate $|S(x,\cdot, \cdot)|_{L^\infty_{2,\xi}\to L^\infty_{3,\xi}}\sim |x|^{-1}$ as $x\to 0$ shows that
the latter holds also at the boundary $p=1$;
we conjecture that it holds for $R$ as well.
This, and the determination of upper bounds \eqref{uni}, $q=\infty$ on $\mathcal{R}$ we regard as very interesting open problems.

\end{document}